\documentclass[12pt]{amsart}

\usepackage{chngcntr}
\usepackage{fullpage} % for larger mirror
\usepackage[normalem]{ulem} % for strikethrough font
\usepackage{xcolor} % for color
\usepackage{booktabs} % for nice tables

\theoremstyle{plain}
\newtheorem{theorem}{Theorem}[section]
\newtheorem{corollary}[theorem]{Corollary}
\newtheorem{lemma}[theorem]{Lemma}
\newtheorem{proposition}[theorem]{Proposition}

\theoremstyle{definition}
\newtheorem{definition}[theorem]{Definition}

\newtheorem{remark}[theorem]{Remark}
\newtheorem{example}[theorem]{Example}
\counterwithin{table}{section}

\def\drop#1{} % uncomment the next line and comment out this line to make dropped material visible, using strikethrough font

\begin{document}

\title{A Sufficient Condition for a  Quandle to be Latin}

\author[Lages]{Ant\'onio Lages}
\email[Lages]{antonio.lages@tecnico.ulisboa.pt}

\author[Lopes]{Pedro Lopes}
\email[Lopes]{pedro.f.lopes@tecnico.ulisboa.pt}

\author[Vojt\v{e}chovsk\'y]{Petr Vojt\v{e}chovsk\'y}
\email[Vojt\v{e}chovsk\'y]{petr.vojtechovsky@du.edu}

\address[Lages, Lopes]{Department of Mathematics, Instituto Superior T\'ecnico, Universidade de Lisboa, Av. Rovisco Pais, 1049-001 Lisbon, Portugal}

\address[Vojt\v{e}chovsk\'y]{Department of Mathematics, University of Denver, 2390 S York St, Denver, Colorado, 80208, USA}

\thanks{Ant\'onio Lages supported by FCT LisMath fellowship PD/BD/150348/2019. Pedro Lopes supported by project FCT PTDC/MAT-PUR/31089/2017, ``Higher Structures and Application''; member of CAMGSD. Petr Vojt\v{e}chovsk\'y supported by Simons Foundation Colaboration Grant for Mathematicians, award no.~855097. Lages and Lopes also supported  by FCT/Portugal through CAMGSD, IST-ID, projects UIDB/04459/2020 and UIDP/04459/2020.}

\begin{abstract}
A quandle is an algebraic structure satisfying three axioms: idempotency, right-invertibility and right self-distributivity. In quandles, right translations are permutations. The profile of a quandle is the list of cycle structures, one per right translation in the quandle. In this note we prove that  if, for each cycle structure in the profile of a quandle, no two cycle lengths are equal, then the quandle is latin---this is the sufficient condition mentioned in the title.
\end{abstract}

\keywords{quandle; permutations; cycle structure; profile; connected quandle; latin quandle}

\subjclass[2020]{20N99}

\maketitle

\section{Introduction}\label{sec:intro}

A quandle is an algebraic structure satisfying three axioms: idempotency, right-invertibility, and right self-distributivity (\cite{Joyce}, \cite{Matveev}). These three axioms constitute an algebraic counterpart to the Reidemeister moves of knot theory (\cite{Kauffman}) and, thus, quandles give rise to efficient ways of distinguishing knots (e.g., via counting colorings, \cite{Dionisio_Lopes}, \cite{Bojarczuk_Lopes}, \cite{FLS}). The algebraic and combinatorial theory of quandles is of interest in its own right.

The second quandle axiom implies that right translations of quandles are permutations. The profile of a quandle (\cite{Lopes_Roseman}) is the list of cycle structures, one cycle structure per right translation in the quandle. In a connected quandle, any two right translations have the same cycle structure and we thus identify its profile with this cycle structure.

In this note we prove that if no cycle structure in the profile of a finite quandle has repeats  i.e., if, for each cycle structure in the profile, no two cycle lengths are equal, then the quandle is latin (Theorem \ref{thm:main})---this is the sufficient condition mentioned in the title. More generally, if a finite quandle contains a right translation, $R_{i_0}$, with cycles of distinct lengths and if all right translations have a unique fixed point, then the corresponding left translation, $L_{i_0}$, is also a permutation and its cycles refine the cycles of $R_{i_0}$ (Proposition \ref{prop:RL}).

The organization of this article is as follows. In Section \ref{sec:background} the background material is introduced, including definitions and facts from the literature. The main results are stated and proved in Section \ref{sec:proofs}. In Section \ref{sctn:further} we collect a few questions for further research. All connected quandles of orders up to $47$ are included in the \texttt{GAP} \cite{GAP} package \texttt{RIG} \cite{Vendramin}. Let us call such connected quandles \emph{small connected quandles}. The appendix contains:
\begin{itemize}
\item the list of small connected quandles whose profiles do not contain repeats (these are among the quandles covered by Theorem \ref{thm:main}),
\item the list of profiles with repeats for small latin quandles (these are among the profiles not covered by Theorem \ref{thm:main} for which the conclusion of the theorem holds),
\item the list of profiles with a single fixed point for small non-latin connected quandles (these are among the profiles that are not immediately detected as not belonging to latin quandles, cf. Lemma \ref{lem:almost}).
\end{itemize}

\section{Background material and results}\label{sec:background}

The algebraic structure known as \emph{quandle}, introduced independently in \cite{Joyce} and \cite{Matveev}, is defined as follows.

\begin{definition}\label{def:quandle}
Let $X$ be a set equipped with a binary operation denoted by $*$. The pair $(X,*)$ is said to be a $\emph{quandle}$ if, for each $i, j, k\in X$,
\begin{enumerate}
\item[(i)] $i*i=i$ (idempotency);
\item[(ii)] $\exists ! x\in X: x*j=i$ (right-invertibility);
\item[(iii)] $(i*j)*k=(i*k)*(j*k)$ (right self-distributivity).
\end{enumerate}
\end{definition}

To each finite quandle $(X,\ast)$ of order $n$ we associate a \emph{quandle table}, which is the $n\times n$ table whose element in row $i$ and column $j$ is $i\ast j$, for all $i,j\in X$. We add an extra $0$-th column (where we display the $i$'s) and an extra $0$-th row (where we display the $j$'s) to quandle tables to improve legibility.

\begin{example}
Table \ref{table:1} is an example of a quandle table for a quandle of order $6$. It is the quandle table for $Q_{6,2}$, see \cite{Lages_Lopes}.
\end{example}

\begin{table}[htb]
\caption{Quandle table for $Q_{6,2}$.}\label{table:1}
\begin{center}
\begin{tabular}{c|c c c c c c}
 $\ast$ & 1 & 2 & 3 & 4 & 5 & 6\\
 \hline
 1 & 1 & 5 & 1 & 6 & 4 & 2\\
 2 & 6 & 2 & 5 & 2 & 1 & 3\\
 3 & 3 & 6 & 3 & 5 & 2 & 4\\
 4 & 5 & 4 & 6 & 4 & 3 & 1\\
 5 & 2 & 3 & 4 & 1 & 5 & 5\\
 6 & 4 & 1 & 2 & 3 & 6 & 6\\
\end{tabular}
\end{center}
\end{table}

For an algebraic structure $(X,*)$ and for $i\in X$, we let $R_i: X\rightarrow X, j\mapsto j*i$ be the \emph{right translation} by $i$ in $(X,*)$ and $L_i: X\rightarrow X, j\mapsto i*j$ be the \emph{left translation} by $i$ in $(X,*)$. For example, the right translation by $1$ in $Q_{6,2}$ is $R_1=(1)(3)(2\,\,6\,\,4\,\,5)$. Note that axioms (ii) and (iii) in Definition \ref{def:quandle} guarantee that each right translation in a quandle is an automorphism of that quandle.

The group $\langle R_i:i\in X\rangle$ generated by the right translations of a quandle $(X,*)$ is said to be the \emph{right multiplication group}. A quandle $(X,*)$ is \emph{connected} if its right multiplication group acts transitively on $X$. A quandle $(X,*)$ is \emph{latin} if all its left translations are permutations of $X$. For instance, $Q_{6,2}$ is connected but not latin.

For a permutation $f$ on a finite set $X$, the \emph{cycle structure} of $f$ is the formal sequence $(1^{c_1}, 2^{c_2}, 3^{c_3},\dots)$ such that $c_m$ is the number of $m$-cycles in a decomposition of $f$ into disjoint cycles. For convenience, we omit entries with $c_i=0$ and we drop $c_i$ whenever $c_i=1$. For example, the right translation by $1$ in $Q_{6,2}$ has cycle structure $(1^2, 4)$.

The \emph{profile} of a quandle is the \drop{sorted }list of the cycle structures of its right translations.

\begin{lemma}\label{lem:automorphism}
Let $(X,*)$ be an algebraic structure and $f$ an automorphism of $(X,*)$. Then $f R_i f^{-1}=R_{f(i)}$, for every $i\in X$.
\end{lemma}
\begin{proof}
Indeed, $fR_i f^{-1}(j)=f(f^{-1}(j)*i)=f(f^{-1}(j))*f(i)=j*f(i)=R_{f(i)}(j)$, for every $i,j\in X$.
\end{proof}

\begin{lemma}\label{lem:connected}
In a finite connected quandle $(X,*)$, all right translations have the same cycle structure.
\end{lemma}

\begin{proof}
Since each right translation in a quandle is an automorphism, this is a straightforward consequence of Lemma \ref{lem:automorphism}.
\end{proof}

We thus define the \emph{profile} $\mathrm{p}(X,*)$ of a finite connected quandle $(X,*)$ to be the cycle structure of any of its right translations. For instance, $Q_{6,2}$ has profile $\mathrm{p}(Q_{6,2})=(1^2,4)$.

\begin{lemma}\label{lem:implication}
Let $(X,*)$ be a latin quandle. Then $(X,*)$ is connected.
\end{lemma}

\begin{proof}
Indeed, $R_{L_i^{-1}(j)}(i)=i*L_i^{-1}(j)=L_i(L_i^{-1}(j))=j$, for every $i,j\in X$.
\end{proof}

Though each latin quandle is connected, not every connected quandle is latin. We would like to know under which assumptions on $\mathrm{p}(X,*)$ a connected quandle
$(X,*)$ is latin. We start with an important observation.

\begin{lemma}\label{lem:almost}
If $(X,*)$ is a latin quandle, then each of its right translations has a unique fixed point.
\end{lemma}

\begin{proof}
Let $i\in X$ and notice that $R_i(i)=i$. Suppose that there is $i\neq j\in X$ such that $R_i(j)=j$. Then $L_j(i)=j*i=R_i(j)=j$ and $L_j(j)=j*j=j$ show that $L_j$ is not a permutation. Hence $(X,*)$ is not latin.
\end{proof}

The simple necessary condition of Lemma \ref{lem:almost} goes a long way towards identifying small connected quandles that are latin. Among the $791$ small connected quandles, there are $547$ quandles that are latin and only $4$ quandles that are not latin and whose right translations have a unique fixed point. These $4$ quandles all happen to be of order $28$ and their profile is either $(1, 3^9)$ or $(1, 3, 6^4)$. Importantly, there are latin quandles of order $28$ with profiles $(1, 3^9)$ and $(1, 3, 6^4)$. Hence, the profile alone cannot serve as both a sufficient and a necessary condition for identifying latin quandles among connected quandles.

Theorem \ref{thm:main} below is a sufficient condition on the profile for a quandle to be latin.

\section{Main results}\label{sec:proofs}

We start with the following results.

\begin{lemma}\label{lem:shift}
Let $f$ be a permutation with a cycle $C=(k,k+1,\dots,k+\ell)$. If $i,j\in C$ then $f^{j-i}(i)=j$.
\end{lemma}
\begin{proof}
This is clear when $j\ge i$. If $j<i$ then $f^{i-j}(j)=i$ and therefore $j = f^{-(i-j)}(i)$.
\end{proof}

\begin{lemma}\label{lem:division}
Let $(X,*)$ be an algebraic structure and $f$ an automorphism of $(X,*)$. The cycle decomposition of $f$ gives rise to a partition  ${\mathcal P}$ of $X$ (whose parts are the disjoint cycles of f). Suppose $x,y,z\in X$ are such that $x*y=z$,  $u\in C_u \in{\mathcal P}$, where $u\in \{x, y, z\}$. Then $\ell_z$ divides $\mathrm{lcm}(\ell_x, \ell_y)$, where $\ell_u=\lvert C_u\rvert$.
\end{lemma}

\begin{proof}
If $x*y=z$ then $f(x)*f(y)=f(x*y)=f(z)$ and, by induction, $f^n(x)*f^n(y)=f^n(z)$, for every $n\geq0$. Let $m=\mathrm{lcm}(\ell_x, \ell_y)$. Then $f^m(x)=x$ and $f^m(y)=y$, so $f^m(z)=f^m(x)*f^m(y)=x*y=z$.
\end{proof}

See \cite[Lemma 3.6]{SVW} for a version of Lemma \ref{lem:division} for autotopisms of latin squares.

\begin{proposition}\label{prop:RL}
Let $(X,*)$ be a finite quandle with $1\in X$. Suppose that the cycles of the right translation $R_1$ have distinct lengths and that every right translation has a unique fixed point. Then the left translation $L_1$ is a permutation of $X$ and the elements of every cycle of $L_1$ are contained in some cycle of $R_1$.
\end{proposition}

\begin{proof}
Let $X=\{1,\dots,n\}$. Suppose that $R_1$ has $c$ cycles of lengths $1=\ell_1<\ell_2<\cdots<\ell_c$. Let $a_0=0$ and $a_i=a_{i-1}+\ell_i$. By renaming the elements of $X$, we can assume without loss of generality that $R_1=(a_0+1,\dots,a_1)(a_1+1,\dots,a_2)\cdots(a_{c-1}+1,\dots,a_c)$. For $1\leq t\leq c$, let $C_t=\{a_{t-1}+1,\dots,a_t\}$ be the underlying set of the $t$-th cycle of $R_1$.

We will first prove that $1*C_t=\{1*x:x\in C_t\}\subseteq C_t$, for every $t$. For $t=1$, we have that $C_1=\{1\}$ and therefore $1*C_1=\{1*1\}=\{1\}\subseteq C_1$. Now, let $t>1$ and suppose that $1*C_s\subseteq C_s$ holds for all $s<t\leq c$. We will prove that $1*C_t\subseteq C_t$. Let $x\in C_t$ and $y\in C_u$ be such that $1*x=y$. By Lemma \ref{lem:division}, $\ell_u$ divides $\mathrm{lcm}(\ell_1,\ell_t)=\mathrm{lcm}(1,\ell_t)=\ell_t$, so $u\in\{1,\dots,t\}$. If $u=t$, we are done. Hence, suppose that $u\in\{1,\dots,t-1\}$. We will show that $z=R_1^{\ell_u}(x)$ is a fixed point of $R_x$. Noting that $z\neq x$ (as $x\in C_t$ and $\lvert C_t\rvert=\ell_t>\ell_u$), this will conflict with $x$ being the unique fixed point of $R_x$.

Note that $R_y(1)\in C_u$ by the induction assumption. Let $k=R_y(1)-y$ and consider the automorphism
\begin{displaymath}
    f=R_y^{-1}R_1^{k}R_x.
\end{displaymath}
Since $R_y(1)\in C_u$ and $y\in C_u$, Lemma \ref{lem:shift} yields $R_1^{k}(y)=R_y(1)$ and therefore
\begin{displaymath}
    f(1)=R_y^{-1}R_1^{k}R_x(1)=R_y^{-1}R_1^{k}(y)=R_y^{-1}(R_y(1))=1.
\end{displaymath}
Also note, using Lemma \ref{lem:automorphism}, that $R_1^{\ell_u}R_y R_1^{-\ell_u}=R_{R_1^{\ell_u}(y)}=R_y$ (since $y\in C_u$ and $\lvert C_u\rvert=\ell_u$), so $R_1^{\ell_u}=R_y R_1^{\ell_u}R_y^{-1}$. Then we have
\begin{align*}
    R_x^{-1}(z) &=R_x^{-1}R_1^{\ell_u}(x)=R_x^{-1}R_1^{\ell_u}R_x(x)\\
        &=R_x^{-1}R_1^{-k}R_1^{\ell_u}R_1^{k}R_x(x)\\
        &=R_x^{-1}R_1^{-k}R_y R_1^{\ell_u}R_y^{-1}R_1^{k}R_x(x)\\
        &=f^{-1}R_1^{\ell_u}f(x) =(f^{-1}R_1 f)^{\ell_u}(x) = R_{f^{-1}(1)}^{\ell_u}(x)=R_1^{\ell_u}(x)=z.
\end{align*}

We are now ready to prove that $L_1$ is surjective, hence bijective. Let $x\in X$, say $x\in C_t=\{a_{t-1}+1,\dots,a_t\}$. By the previous paragraph, there is $y\in C_t$ such that $1*a_t=y$. Since both $x$ and $y$ belong to $C_t$, there is $1\leq m\leq \ell_t$ such that $x=R_1^m(y)$. Then we have $L_1(a_{t-1}+m)=1*(a_{t-1}+m)=R_1^m(1)*R_1^m(a_t)=R_1^m(1*a_t)=R_1^m(y)=x$.

Finally, since $1*C_t\subseteq C_t$ for every underlying set $C_t$ of a cycle of $R_1$, it follows that the elements of every cycle of $L_1$ are contained in some cycle of $R_1$.
\end{proof}

\begin{theorem}\label{thm:main}
Let $(X,*)$ be a finite quandle such that each right translation of $(X,*)$ has cycles of distinct lengths. Then $(X,*)$ is a latin quandle.
\end{theorem}
\begin{proof}
Since each right translation has cycles of distinct lengths and it has a fixed point, it has precisely one fixed point. Given any right translation $R_i$ of $(X,*)$, Proposition \ref{prop:RL} then shows that the corresponding left translation $L_i$ is a permutation of $X$. Hence $(X,*)$ is a latin quandle.
\end{proof}

As an immediate corollary of Theorem \ref{thm:main}, we obtain:

\begin{corollary}\label{cor:connected}
Let $(X,*)$ be a finite connected quandle whose profile $\mathrm{p}(X,*)$ has no repeats. Then $(X,*)$ is a latin quandle.
\end{corollary}

\begin{example}
The quandle $Q_{9,4}$, whose quandle table is given in Table \ref{table:3}, is a connected quandle of order $9$ with profile $(1,2,6)$. Since $1<2<6$, we conclude, by Theorem \ref{thm:main}, that this quandle is latin, as we can easily check by inspecting Table \ref{table:3}.
\end{example}

\begin{table}[htb]
\caption{Quandle table for $Q_{9,4}$.}\label{table:3}
\begin{center}
\begin{tabular}[htb]{c|c c c c c c c c c}
 $\ast$ & 1 & 2 & 3 & 4 & 5 & 6 & 7 & 8 & 9\\
 \hline
 1 & 1 & 3 & 2 & 9 & 8 & 7 & 6 & 5 & 4\\
 2 & 3 & 2 & 1 & 8 & 7 & 9 & 5 & 4 & 6\\
 3 & 2 & 1 & 3 & 7 & 9 & 8 & 4 & 6 & 5\\
 4 & 7 & 9 & 8 & 4 & 6 & 5 & 1 & 3 & 2\\
 5 & 9 & 8 & 7 & 6 & 5 & 4 & 3 & 2 & 1\\
 6 & 8 & 7 & 9 & 5 & 4 & 6 & 2 & 1 & 3\\
 7 & 5 & 4 & 6 & 2 & 1 & 3 & 7 & 9 & 8\\
 8 & 4 & 6 & 5 & 1 & 3 & 2 & 9 & 8 & 7\\
 9 & 6 & 5 & 4 & 3 & 2 & 1 & 8 & 7 & 9\\
\end{tabular}
\end{center}
\end{table}

\begin{example}
By Theorem \ref{thm:main}, every connected quandle with profile of the form $(1,\ell)$, where $1<\ell$, is latin. These are the so-called \emph{quandles of cyclic type}, see \cite{Lopes_Roseman}.
\end{example}

\begin{remark}
(i) Note that the conclusion of Proposition \ref{prop:RL} cannot be strengthened to assert that the elements of every cycle of $L_1$ coincide with the elements of some cycle of $R_1$. For instance, in the quandle $Q_{9,4}$ we have $R_1=(1)(2,3)(4,7,5,9,6,8)$ and $L_1=(1)(2,3)(4,9)(5,8)(6,7)$.

(ii) The assumption that all right translations have a unique fixed point is necessary in Proposition \ref{prop:RL}. In the quandle with quandle table
\begin{center}
\begin{tabular}{c|ccc}
 $\ast$&1&2&3\\
 \hline
 1&1&1&1\\
 2&3&2&2\\
 3&2&3&3\\
\end{tabular}
\end{center}
the right translation $R_1=(1)(2,3)$ has cycles of distinct lengths but the left translation $L_1$ is not a permutation.

(iii) We are not aware of a quandle that satisfies the assumptions of Proposition \ref{prop:RL} but is not connected.
\end{remark}

See the appendix for profiles of small connected quandles in relation to Theorem \ref{thm:main}.

\section{Further research}\label{sctn:further}

The connected quandles of order 28 demonstrate that we cannot identify latin quandles among connected quandles by their profile alone. Among the $791$ small connected quandles, there are $547$ latin quandles, of which $183$ have cycles of distinct lengths and $364$ do not. Can the sufficient condition of Theorem \ref{thm:main} be generalized to capture more latin quandles among connected quandles? Is there a criterion for characterizing latin quandles among connected quandles based on the profile and other combinatorial properties?

A conjecture of Hayashi \cite{Hayashi} states that every right translation of a finite connected quandle has a regular cycle, that is, it is a permutation whose order is equal to the length of its longest cycle. Is there a connection between Hayashi's conjecture and Theorem \ref{thm:main}? In particular, can Theorem \ref{thm:main} be generalized under the assumption that Hayashi's conjecture is true?

\section*{Acknowledgement}

Corollary \ref{cor:connected} was initially obtained by the first two authors. Their proof was later simplified and generalized by the third author to yield Theorem \ref{thm:main}.

\section{Appendix}

\emph{In this appendix we omit the unique fixed point from all profiles.} Table \ref{table:5} lists all small connected quandles whose profile contains no repeats---these are the small quandles covered by Theorem \ref{thm:main}. Table \ref{table:6} lists profiles of all small latin quandles whose profiles contain repeats---these are the profiles of small latin quandles not covered by Theorem \ref{thm:main}. Finally, there are precisely two profiles of small non-latin connected quandles that have a unique fixed point, namely $(3^9)$ and $(3,6^4)$. Both of these profiles also appear in Table \ref{table:6}.

\begin{table}[htb]
\caption{Small connected quandles $Q_{n,m}$ whose profile contains no repeats. The unique fixed point is omitted from all profiles.}\label{table:5}
\begin{small}
\begin{tabular}{@{}lll@{\hskip 4mm}|@{\hskip 4mm}lll@{}}
\toprule
$n$ &$m$                            &$\mathrm{p}(Q_{n,m})$               &$n$  &$m$                                  &$\mathrm{p}(Q_{n,m})$\\
\midrule
$1$ & $1$                           & $()$                               &$24$ & $24,25$                             & $(2,7,14)$ \\
$3$ & $1$                           & $(2)$                              &$25$ & $21$--$30$                          & $(4,20)$ \\
$4$ & $1$                           & $(3)$                              &     & $31$--$34$                          & $(24)$ \\
$5$ & $2,3$                         & $(4)$                              &$27$ & $37$--$40,47$--$52,60,61$           & $(2,6,18)$  \\
$7$ & $4,5$                         & $(6)$                              &     & $62$--$65$                          & $(26)$ \\
$8$ & $2,3$                         & $(7)$                              &$29$ & $1,2,7,9,10,13,14,17,18,20,25,26$   & $(28)$ \\
$9$ & $4$--$6$                      & $(2,6)$                            &$31$ & $2,10$--$12,16,20,21,23$            & $(30)$ \\
    & $7,8$                         & $(8)$                              &$32$ & $10$--$15$                          & $(31)$ \\
$11$& $6$--$9$                      & $(10)$                             &     & $16,17$                             & $(3,7,21)$ \\
$12$& $4$                           & $(2,3,6)$                          &$36$ & $69,70$                             & $(3,8,24)$  \\
$13$& $1,5,6,10$                    & $(12)$                             &$37$ & $24$--$35$                          & $(36)$ \\
$16$& $8,9$                         & $(15)$                             &$40$ & $29$--$32$                          & $(4,7,28)$ \\
$17$& $2,4$--$6,9$--$11,13$         & $(16)$                             &$41$ & $24$--$39$                          & $(40)$ \\
$19$& $1,2,9,12$--$14$              & $(18)$                             &$43$ & $30$--$41$                          & $(42)$ \\
$20$& $7,8$                         & $(3,4,12)$                         &$44$ & $6$--$9$                            & $(3,10,30)$  \\
$23$& $4,6,9,10,13,14,16,18$--$20$  & $(22)$                             &$47$ & $24$--$45$                          & $(46)$ \\
\bottomrule
\end{tabular}
\end{small}
\end{table}

\begin{table}[htb]
\caption{Profiles of small latin quandles that contain repeats. The unique fixed point is omitted from all profiles.}\label{table:6}
\begin{small}
\begin{tabular}{@{}l@{\hskip 8mm}llllllll@{}}
\toprule
$n$ &profiles\\
\midrule
$5$ &$(2^2)$\\
$7$ &$(2^3)$         & $(3^2)$\\
$9$ &$(2^4)$         & $(4^2)$\\
$11$&$(2^5)$         & $(5^2)$\\
$13$&$(2^6)$         & $(3^4)$      & $(4^3)$         & $(6^2)$\\
$15$&$(2^7)$         & $(2,4^3)$    & $(2^2,10)$\\
$16$&$(3^5)$         & $(5^3)$      & $(3,6^2)$\\
$17$&$(2^8)$         & $(4^4)$      & $(8^2)$\\
$19$&$(2^9)$         & $(3^6)$      & $(6^3)$         & $(9^2)$\\
$20$&$(2^2,3,6^2)$\\
$21$&$(2^{10})$      & $(2,6^3)$    & $(2,3^2,6^2)$   & $(2^3,14)$\\
$23$&$(2^{11})$      & $(11^2)$\\
$25$&$(2^{12})$      & $(3^8)$      & $(4^6)$         & $(2^2,4^5)$     & $(6^4)$      & $(8^3)$     & $(2^2,10^2)$    & $(12^2)$\\
$27$&$(2^{13})$      & $(2,4^6)$    & $(2^4,6^3)$     & $(2,6^4)$       & $(2,8^3)$    & $(13^2)$    & $(2^4,18)$\\
$28$&$(3^9)$         & $(3,6^4)$    & $(2^3,3,6^3)$\\
$29$&$(2^{14})$      & $(4^7)$      & $(7^4)$         & $(14^2)$\\
$31$&$(2^{15})$      & $(3^{10})$   & $(5^6)$         & $(6^5)$         & $(10^3)$     & $(15^2)$\\
$33$&$(2^{16})$      & $(2^5,22)$   & $(2,5^2,10^2)$  & $(2,10^3)$\\
$35$&$(3^2,4,12^2)$  & $(4,6,12^2)$ & $(2^2,3^2,6^4)$ & $(2^2,6^5)$     & $(2^3,4^7)$  & $(2^{17})$\\
$36$&$(3,4^2,12^2)$  & $(2^4,3,6^4)$& $(2,3,6^5)$,\\
$37$&$(2^{18})$      & $(3^{12})$   & $(4^9)$         & $(6^6)$         & $(9^4)$      & $(12^3)$    & $(18^2)$\\
$39$&$(2^{19})$      & $(2^6,26)$   & $(2,12^3)$      & $(2,6^6)$       & $(2,4^9)$    & $(2,3^4,6^4)$\\
$40$&$(2^2,7,14^2)$\\
$41$&$(2^{20})$      & $(4^{10})$   & $(5^8)$         & $(8^5)$         & $(10^4)$     & $(20^2)$\\
$43$&$(2^{21})$      & $(3^{14})$   & $(6^7)$         & $(7^6)$         & $(14^3)$     & $(21^2)$\\
$44$&$(3,5^2,15^2)$  & $(2^5,3,6^5)$\\
$45$&$(4,8^5)$       & $(2^2,8^5)$  & $(4^{11})$      & $(2^2,4^{10})$  & $(2^{22})$   & $(2^4,6^6)$ & $(2^4,4^9)$     & $(2^7,6^5)$\\
    & $(2^7,30)$      & $(2^7,10^3)$    & $(2,4^3,6,12^2)$\\
$47$&$(2^{23})$      & $(23^2)$\\
\bottomrule
\end{tabular}
\end{small}
\end{table}

\end{document}